\documentclass[12pt]{amsart}%
\usepackage{amsfonts}
\usepackage{latexsym}
\usepackage{amssymb}
\usepackage{amsmath}
\usepackage{amstext}
\usepackage{amsthm}
\usepackage{amsxtra}
\usepackage[utf8]{inputenc}
\usepackage[brazilian,english]{babel}
\providecommand{\U}[1]{\protect\rule{.1in}{.1in}}
\newtheorem{The}{Theorem}
\newtheorem{Cor}[The]{Corollary}
\newtheorem{Pro}[The]{Proposition}
\newtheorem{Lem}[The]{Lemma}
\numberwithin{equation}{section}
\numberwithin{The}{section}
\usepackage{enumerate}

\def\({\left(}       \def\){\right)}

\begin{document}
	
\title{Level sets of the hyperbolic derivative for analytic self-maps of the unit disk}

\author[J. Arango]{Juan Arango}
\address{Escuela de Matem\'aticas, Universidad Nacional de Colombia, Medell\'{\i}n, Colombia}
\email{jharango@unal.edu.co}

\author[H. Arbel\'aez]{Hugo Arbel\'{a}ez}
\address{Escuela de Matem\'aticas, Universidad Nacional de Colombia, Medell\'{\i}n, Colombia}
\email{hjarbela@unal.edu.co}

\author[D. Mej\'ia]{Diego Mej\'ia}
\address{Escuela de Matem\'aticas, Universidad Nacional de Colombia, Medell\'{\i}n, Colombia}
\email{dmejia@unal.edu.co}

	\maketitle
	
	\begin{abstract}
		Let the function $\varphi$ be holomorphic in the unit disk $\mathbb{D}$ of the complex plane $\mathbb{C}$ and let $\varphi (\mathbb{D})\subset \mathbb{D}$. We study the level sets and the critical points of the hyperbolic derivative of $\varphi$, $$|D_{\varphi}(z)|:=\frac{(1-|z|^2)|\varphi'(z)|}{1-|\varphi(z)|^2}.$$ In particular, we show how the Schwarzian derivative of $\varphi$ reveals the nature of the critical points.\\\\\\
		\textbf{Keywords}: Hyperbolic derivative, Level set, Critical point, Trajectory, Schwarzian derivative, Hyperbolic order.\\\\\\
		\textbf{Mathematics Subject Classification}: 30H05, 30E99, 30J99. 
	\end{abstract}
	\section{Introduction}\label{sec1}
	Let $\varphi $ be analytic and not constant in the unit disk $\mathbb{D}$ and let $\varphi (\mathbb{D})\subset \mathbb{D}$. We will study the \textit{level sets} of the \textit{hyperbolic derivative} of $\varphi$, namely the sets 
	\begin{equation*}
		C_{\varphi}(t):=\Big \{ z\in \mathbb{D}:(1-\left\vert z\right\vert ^{2})
		\frac{\left\vert \varphi^{\prime}(z)  \right\vert }{1-\left\vert
			\varphi (z)  \right\vert ^{2}}=t \Big \}, \;\;\; 0\leq t< 1.
	\end{equation*}
	In \cite{Po08} Pommerenke studied the level sets
	\begin{equation*}
		\Big \{ z\in \mathbb{D}:(1-\left\vert z\right\vert ^{2})
		\left\vert f^{\prime}(z)  \right\vert =t \Big \}, \;\;\; 0<t<+\infty,
	\end{equation*}
	for an analytic function $f$ defined in $\mathbb{D}$. More recently, in \cite {AAM17}, the authors  considered meromorphic functions $g$ in $\mathbb{D}$ and studied the level sets
	\begin{equation*}
		\Big \{ z\in \mathbb{D}:(1-\left\vert z\right\vert ^{2})
		\frac{\left\vert g^{\prime}(z)  \right\vert }{1+\left\vert
			g(z)  \right\vert ^{2}}=t \Big \}, \;\;\; 0<t<+\infty.
	\end{equation*} 
	Two differential operators will play an important role in this work:
	\begin{equation*}
		D_{\varphi}(z):=(1-|z|^2)\frac{\varphi^{\prime}(z)}{1-\left\vert
			\varphi (z)  \right\vert ^{2}}
	\end{equation*}
	and
	\begin{align}\label{eq1.2}
		A_{\varphi}(z):& =(1-|z|^2)\frac{\partial}{\partial z}\log |D_{\varphi} (z)|=\frac{1-|z|^2}{|D_{\varphi} (z)|}\frac{\partial|D_{\varphi} (z)|}{\partial z}\\
		&=\frac{1-|z|^2}{2}\frac{\varphi ''(z)}{\varphi '(z)}-\bar{z}+\frac{(1-|z|^2)\overline{\varphi(z)}\varphi'(z)}{1-|\varphi(z)|^2},\notag	
	\end{align}
	being $\infty$ at the zeros of $\varphi '$.\\\\
	The first order operator satisfies the chain rule
	\begin{equation*}
		D_{\varphi \circ \psi}(z)=D_{\varphi}(\psi(z))D_{\psi}(z).
	\end{equation*}
	In particular, if $\tau$ belongs to the group $\textup{M\"ob}(\mathbb{D})$ of conformal automorphisms of the unit disk, then
	\begin{equation*}
		D_{\varphi \circ \tau}(z)=\frac{\tau '(z)}{|\tau '(z)|}D_{\varphi}(\tau (z));\;\;\;D_{\tau \circ \varphi}(z)=\frac{\tau '(\varphi (z))}{|\tau '(\varphi (z))|}D_{\varphi}(z).
	\end{equation*}
	\noindent The second order differential operator
	verifies the chain rule
	\begin{equation*}
	A_{\varphi \circ \psi}(z)=\frac{1-|z|^2}{1-|\psi (z)|^2}A_{\varphi}(\psi (z))\psi '(z)+A_{\psi}(z),
	\end{equation*}
	at points $z$ for which $\psi'(z)$ and $\varphi'(\psi(z))$ are not zero. It can be seen (see Proposition \ref{prop1}) that $A_\varphi\equiv 0$ in a neigborhood of a point if and only if $\varphi \in \textup{M\"ob}(\mathbb{D})$; hence, if $\tau \in \textup{M\"ob}(\mathbb{D})$, then
	\begin{equation}\label{eq1.3}
	A_{\varphi \circ \tau}(z)= \frac{\tau'(z)}{|\tau' (z)|^2}A_{\varphi}(\tau (z));\;\;\;A_{\tau \circ \varphi}(z)=A_{\varphi}(z).
	\end{equation}
	With a different notation, these differential operators have been used by Ma and Minda (see for instance \cite{MaMi95}). For $\varphi$ locally univalent, Ma and Minda defined in \cite {MaMi94} the \textit {(upper) hyperbolic (linearly invariant) order} of $\varphi$ by
	\begin{equation*}
	\alpha(\varphi):=\sup_{z\in \mathbb{D}}|A_{\varphi}(z)|,
	\end{equation*}
	which is, clearly, $+\infty$ when $\varphi'$ has zeros.
	We can also define the \textit{lower hyperbolic order} of $\varphi$ by
	\begin{equation*}
	\mu(\varphi):=\inf_{z\in \mathbb{D}}|A_{\varphi}(z)|.
	\end{equation*}
	The invariance of these orders is a consequence of \eqref{eq1.3}; that is,
	\begin{equation*}
	\alpha(\sigma \circ \varphi \circ \tau)=\alpha(\varphi)\;\; \textup{and}\;\; \mu(\sigma \circ \varphi \circ \tau)=\mu(\varphi)
	\end{equation*}
	for all $\sigma,\tau \in \textup{M\"ob}(\mathbb{D})$.\\\\
	This paper is organized as follows. In Section \ref{sec2} we define the orthogonal trajectories to the level sets as solutions of an initial value problem. It is shown that the hyperbolic lower order is always zero. Also, for $\varphi$ locally univalent, if $A_\varphi$ never vanishes and all trajectories are defined in the same interval, we prove that the level sets $C_\varphi(t)$ are open Jordan arcs. In Section \ref{sec3} we study the critical points of $|D_\varphi|$. The main result in this section is Theorem \ref{th3}, where we show how the Schwarzian derivative reveals the nature of these critical points. Finally, in Section \ref{sec4} we give several examples to illustrate our results.

	\section{Orthogonal trajectories to the level sets}\label{sec2}
	We want to define a family of trajectories orthogonal to the level sets of $|D_\varphi|$. To this end, let $E:=\{z\in \mathbb{D}:\;\varphi'(z),\,A_\varphi(z) \neq 0\}$. From \eqref{eq1.2}, the gradient of $|D_\varphi|$ at $z_0\in E$ verifies
	\begin{align*}
		\nabla |D_\varphi|(z_0)&= 2\overline{\frac{\partial|D_\varphi|}{\partial z}}(z_0)=\frac{2|D_\varphi(z_0)|\overline{A_\varphi(z_0)}}{1-|z_0|^2}\\
		&=\frac{2|D_\varphi(z_0)||A_\varphi(z_0)|^2}{(1-|z_0|^2)A_\varphi(z_0)}.
	\end{align*}
	Suppose that $c(t)$, $0<t<1$, is a positive smooth function. Then, the initial value problem
	\begin{equation}\label{eq2.1}
		\left\{
		\begin{array}
			[c]{c}%
			z^{\prime}(t)=\dfrac{c(t)}{A_\varphi(z(t))}\\
			z(t_{0})=z_{0}
		\end{array}
		\right.
	\end{equation}
	has a unique solution, near $z_0$, orthogonal to the level sets $C_\varphi (t)$. We would like to choose $c(t)$ in such a way that the parameter $t$ of the solution $z(t)$ coincides with the level of the point; that is, 
	\begin{equation}\label{eq2.2}
	|D_\varphi(z(t))|=\frac{(1-|z(t)|^2)|\varphi'(z(t))|}{1-|\varphi(z(t))|^2}=t.
	\end{equation}
	 By \eqref{eq2.1},
	\begin{align*}
		\frac{1}{t}&=\dfrac{d}{dt}\log t=\frac{d}{dt}\log|D_\varphi (z(t))|\\
		&=2\operatorname{Re}\frac{A_\varphi(z(t))z'(t)}{1-|z(t)|^2}\\
		&=\frac{2c(t)}{1-|z(t)|^2}.
	\end{align*}
	Hence, $c(t)=\dfrac{1-|z(t)|^2}{2t}$.\\\\
	Now, we redefine the initial value problem \eqref{eq2.1}. Let us consider the function $G(t,z):=\dfrac{1-|z|^{2}}{2tA_{\varphi}(z)}$ in the domain $\Omega:=(0,1)\times E$. By the general theory of ordinary differential equations \cite{So79}, the initial value problem
	\begin{equation}\label{eq2.3}
		\left\{
		\begin{array}
			[c]{c}%
			z^{\prime}(t)=G(t,z\left(  t\right)  )\\
			z(t_{0})=z_{0}%
		\end{array}
		\right.
	\end{equation}
	has a unique maximal solution $z(t)$, defined for $t \in (\omega^-,\omega^+)$, such that $(t,z(t))$ tends to $\partial \Omega=\big ([0,1]\times \partial E\big )\cup \big (\{0,1\}\times \mathbb{D}\big )$ when $t\to \omega ^{\pm}$. This solution is called the \textit{trajectory} of $\varphi$ through the point $z_0$. Since the parameter $t$ is the level of the point $z(t)$, then $(t,z(t))$ does not accumulate on $\{0,1\}\times E$. Therefore, the trajectory $z(t)$ constitutes an open Jordan arc $\Gamma$ with $\overline{\Gamma}\setminus \Gamma\subset \partial E$. We write 
	\begin{equation}\label{eq2.4}
		\Gamma^+:=\{z(t): t_0\leq t<\omega^+\}\;\;\; \text{and}\;\;\; \Gamma^-:=\{z(t): \omega^-<t\leq t_0\}.
	\end{equation}
	\begin{The}\label{th1}
		For any trajectory $\Gamma$, $\inf_{z\in \Gamma^+}|A_\varphi(z)|=0$.	
	\end{The}
	\begin{proof}
		For $t_0<t<\omega^+$,
		\begin{align}\label{eq2.5}
		\log\frac{t}{t_0}=\int_{t_0}^{t}\frac{1}{s}ds&=\int_{t_0}^{t}\frac{2|A_{\varphi}(z(s))z^{\prime}(s)|}{1-|z(s)|^2}ds\notag\\
		&\geq 2\inf_{z\in \Gamma^+}|A_{\varphi}(z)|\int_{t_0}^{t}\frac{|z^{\prime}(s)|}{1-|z(s)|^2}ds\notag\\
		&\geq 2\inf_{z\in \Gamma^+}|A_{\varphi}(z)|d_{\mathbb{D}}(z(t_0),z(t)),
		\end{align}
		where $d_{\mathbb{D}}$ is hyperbolic distance in $\mathbb{D}$.\\
		Take any increasing sequence $\{t_n\}$ in $(t_0,\omega^+)$ converging to $\omega^+$ such that $\{z(t_n)\}$ converges to a point $\zeta$. Since $\overline{\Gamma^+}\setminus\Gamma^+\subset \partial E$ then either $\zeta\in \mathbb{T}$, in which case $d_{\mathbb{D}}(z(t_0),z(t_n))\to + \infty$ as $t_n\to \omega^+$, or else $A_{\varphi}(\zeta)=0,\infty$. But, $\omega^+=\lim t_n=\lim|D_\varphi(z(t_n))|=|D_\varphi(\zeta)|$; so, $A_{\varphi}(\zeta)$ cannot be infinite. In any case,  it follows from \eqref{eq2.5} that $\inf_{z\in \Gamma^+}|A_{\varphi}(z)|=0$.
	\end{proof}
	\noindent As a direct consequence, we have
	\begin{Cor}\label{cor1}
		$\mu (\varphi)=0$.
	\end{Cor}
	\noindent\textbf{Remark.} This corollary shows  that, in the hyperbolic setting, the lower order does not behave as in the Euclidean and spherical frames (see \cite{CrPo07} and \cite{AAM17} for more details).\\\\
	The following theorem gives conditions under which the level sets are connected (see also Example 1, Section 4).  Let us write
	\[
	c:=\inf_{z\in \mathbb{D}}|D_{\varphi}(z)|,\;\;\;\; d:=\sup_{z\in \mathbb{D}}|D_{\varphi}(z)|
	\]
	and 
	\[
	B(t)=\{z\in\mathbb{D}:|D_\varphi(z)|\geq t\}=\bigcup_{s\geq t}C_\varphi(s)\;\;\text{ for all}\;\; t\in (c,d).
	\]
	\begin{The}\label{th2}
		Suppose that $\varphi$ is locally univalent and $A_{\varphi}(z)\neq 0$ for all $z\in \mathbb{D}$. If $\omega^{-}=c$ and $\omega^{+}=d$ for all trajectories, then $C_{\varphi}(t)$ is an open Jordan arc for all $t\in (c,d)$. Furthermore, $B:=\bigcap_{t}\overline{B(t)}$ is a compact connected subset of $\mathbb{T}$ such that any trajectory ends on $B$.
	\end{The}
	\begin{proof}
	Our reasoning follows the argument given by Pommerenke in \cite{Po08}, Theorem 2.1 (see also \cite{AAM17}, Theorem 2). Fix $c<\tau<d$ and write $C(\tau):=C_{\varphi}(\tau)$. Since 
	\[
	\nabla |D_{\varphi}(z)|=2\frac{|\varphi '(z)|\overline{A_{\varphi}(z)}}{1-|\varphi(z)|^2}\neq 0,
	\]
	then $C(\tau)$ is a union of Jordan arcs. Furthermore, since a trajectory cannot cross $C(\tau)$ at two different points, then no component $C$ of $C(\tau)$ is a Jordan curve inside $\mathbb{D}$, so that $\overline{C}\setminus C\subset \mathbb{T}$.\\
	Now, we prove that $C(\tau)$ is connected. Let $w_1,w_2\in C(\tau)$; fix an arbitrary point $w$ in the segment $[w_1,w_2]$. The solution of the initial value problem
	\begin{align*}
	z^{\prime}(t)=\dfrac{1-|z(t)|^2}{2(t+|D_\varphi(w)|)A_\varphi(z(t))},\;\;z(0)=w%
	\end{align*}
	is $\gamma_w(t)=\beta_w(t+|D_\varphi(w)|)$, where $\beta_w$ is the trajectory through $w$ given by \eqref{eq2.3}. The continuous dependence of $\gamma_w(t)$ with respect to $w$ implies that the function $h(w):=\gamma_w(\tau-|D_\varphi(w)|)$ is continuous on the segment $[w_1,w_2]$. Since $h(w)=\beta_w(\tau)$, by \eqref{eq2.2}, $h(w)\in C(\tau)$; therefore $h([w_1,w_2])$ is a connected subset of $C(\tau)$ that joins $w_1=h(w_1)$ and $w_2=h(w_2)$.\\
	To prove that the set $B$ is connected, fix $t\in (c,d)$ and take a trajectory $\Gamma$ parameterized by $z(s)$. Then $\Gamma^+:=\{z(s):s\geq t\}$ is a connected subset of $B(t)$ which meets $C_\varphi(s)$ for each $s\in[t,d)$. Hence $B(t)$ is connected and therefore $B$ is an intersection of nested compact connected sets, clearly contained in $\mathbb{T}$.\\
	Finally, let us show that $\overline{\Gamma^+}\setminus \Gamma^+\subset B$. Take $z_0\in\overline{\Gamma^+}\setminus \Gamma^+$ and $s\in (c,d)$. We may assume the existence of a sequence $t_n \in [t,d)$ converging to $t_0$ such that $z_n:=z(t_n)\to z_0$. Hence $t_0=d$; otherwise, $z_0=z(t_0)$ would be in $\Gamma^+$. Then, for sufficiently large $n$, $s<t_n$ and therefore $z_n\in B(s)$ and so $z_0\in\overline{B(s)}$.
	\end{proof}
	\textit{Conjecture}. The compact connected set $B$ is a single point.\\\\
	\textbf{Remark.} The corresponding result to the previous theorem in the Euclidean and spherical frames (see \cite{Po08}, \cite{AAM17}) also holds, weakening the hypothesis that the lower order is positive. We would like to point out that in \cite{Po08} Pommerenke shows that the set $B$ is a single point. However, he offers no proof to discard the possibility of a wiggling behavior of the trajectory near $\mathbb{T}$.  

	\section{Critical points of $|D_\varphi|$}\label{sec3}
	The \textit{critical points} of $|D_{\varphi}(z)|$ are the discrete zeros of $\varphi '$ and the not necessarily discrete zeros of $A_{\varphi}(z)$, as we can see from \eqref{eq1.2}. Indeed, $A_\varphi$ could be identically equal to zero as is shown in the next proposition.
	\begin{Pro}\label{prop1}
		$A_{\varphi}\equiv 0$ in a neighborhood of a point if and only if $\varphi \in \textup{M\"ob}(\mathbb{D})$.	
	\end{Pro}
	\begin{proof}
		Suppose that $A_\varphi(z)=0$ for $z$ in an open disk $\Delta\subset \mathbb{D}$. By \eqref{eq1.2},
		\[
		\frac{\varphi''(z)}{\varphi'(z)}=\frac{2\bar{z}}{1-|z|^2}-\frac{2\varphi'(z)\overline{\varphi(z)}}{1-|\varphi(z)|^2},\;\;\;z\in \Delta.
		\]
		Then, taken partial derivatives with respect to $\bar{z}$, we have
		\[
		0=\frac{1}{(1-|z|^2)^2}-\frac{|\varphi'(z)|^2}{(1-|\varphi(z)|^2)^2}.
		\]
		Now, it follows from Schwarz-Pick's lemma that $\varphi\in \operatorname{M\ddot{o}b}(\mathbb{D})$.\\\\
		Conversely, if $\varphi\in \operatorname{M\ddot{o}b}(\mathbb{D})$, then a straightforward computation shows that $A_\varphi \equiv 0$ in $\mathbb{D}$. 
		
	\end{proof}
	\noindent The function $1/|D_{\varphi}|$ is subharmonic in $\{z\in\mathbb{D}:\;\varphi'(z)\neq0\}$ and so, the local minima of $|D_{\varphi}|$ occur at points $z$ where $\varphi'(z)=0$. This is a consequence of the following proposition.
	\begin{Pro}\label{prop2}
		Let $\varphi:\mathbb{D} \to \mathbb{D}$ be analytic and not constant. If $\varphi '(z_0)\neq 0$, then
		\begin{equation*}
		\dfrac{\partial^2\log |D_{\varphi}|}{\partial z \partial \bar{z}}(z_0)=-\dfrac{1-|D_{\varphi}(z_0)|^2}{(1-|z_0|^2)^2}.
		\end{equation*}
	\end{Pro}
	\begin{proof}
		From \eqref{eq1.2}
		\begin{align*}
		\frac{\partial ^2}{\partial \bar{z}\partial z}\log |D_{\varphi} (z)| & = \frac{\partial }{\partial \bar{z}}\Big [\frac{A_{\varphi}(z)}{1-|z|^2}\Big ] \\
		& = \frac{1}{(1-|z|^2)^2}\bigg [(1-|z|^2)\Big [-\frac{1}{2}z\frac{\varphi''(z)}{\varphi '(z)}-1+\frac{(1-|z|^2)|\varphi '(z)|^2}{(1-|\varphi(z)|^2)^2}\\
		&\hspace{2.3cm}-\frac{z\overline{\varphi(z)}\varphi '(z)}{1-|\varphi(z)|^2} \Big]+zA_{\varphi}(z)\bigg]\\
		& = \frac{1}{(1-|z|^2)^2}\bigg [-zA_{\varphi}(z)-1+\frac{(1-|z|^2)^2|\varphi '(z)|^2}{(1-|\varphi(z)|^2)^2}+zA_{\varphi}(z)\bigg ]\\
		& = -\dfrac{1}{(1-|z|^2)^2}\big [1-|D_{\varphi}(z)|^2\big ].
		\end{align*}
	\end{proof}
	Since the trajectories $\Gamma$ may go to these points, then it is relevant to investigate the behavior of the hyperbolic derivative near the critical points. Theorem \ref{th3} gives a partial answer. In order to prove it  we will expand $|D_\varphi(z)|$ around $z_0\in E$.\\
	For $z\to z_0$ we have
	\begin{align*}
	|D_{\varphi}(z)|&=|D_{\varphi}(z_0)|+\frac{\partial|D_{\varphi}(z_0)|}{\partial z}(z-z_0)+\frac{\partial|D_{\varphi}(z_0)|}{\partial \bar{z}}(\bar{z}-\bar{z_0})\\ 
	&+\frac{1}{2}\frac{\partial^2|D_{\varphi}(z_0)|}{\partial z^2}(z-z_0)^2+\frac{1}{2}\frac{\partial^2|D_{\varphi}(z_0)|}{\partial \bar{z}^2}(\bar{z}-\bar{z_0})^2\\ 
	&+\frac{\partial^2|D_{\varphi}(z_0)|}{\partial z \partial \bar{z}}|z-z_0|^2+O(|z-z_0|^3).
	\end{align*}
	From \eqref{eq1.2}
	\begin{equation*}
	\frac{\partial|D_{\varphi}|}{\partial\bar{z}}=\overline{\frac{\partial|D_{\varphi}|}{\partial z}}=\frac{|D_{\varphi}|\overline{A_{\varphi}}}{1-|z|^2};
	\end{equation*}
	also
	\begin{equation*}
	\frac{\partial^2|D_{\varphi}|}{\partial \bar{z}^2}=\frac{\partial}{\partial\bar{z}}\Big (\overline{\frac{\partial|D_{\varphi}|}{\partial z}}\Big )=\overline{\frac{\partial}{\partial z}\Big (\frac{|D_{\varphi}|}{\partial z}\Big )}=\overline{\frac{\partial^2|D_{\varphi}|}{\partial z^2}},
	\end{equation*}
	then we obtain
	\begin{align}\label{eq3.1}
	|D_{\varphi}(z)|&=|D_{\varphi}(z_0)|+2\frac{|D_{\varphi}(z_0)|}{1-|z_0|^2}\operatorname{Re}\big (A_{\varphi}(z_0)(z-z_0)\big )\notag\\ 
	&+\operatorname{Re} \Big (\frac{\partial^2|D_{\varphi}(z_0)|}{\partial z^2}(z-z_0)^2 \Big )+\frac{\partial^2|D_{\varphi}(z_0)|}{\partial z \partial \bar{z}}|z-z_0|^2+O(|z-z_0|^3).
	\end{align}
	On the other hand, from \eqref{eq1.2},
	\begin{align*}
	\frac{\partial A_{\varphi}}{\partial z} & = \frac{1-|z|^2}{2}\frac{\varphi '''}{\varphi '}-\frac{1-|z|^2}{2}\Big (\frac{\varphi ''}{\varphi '} \Big )^2 - \frac{\bar{z}}{2}\frac{\varphi ''}{\varphi '}-\frac{\bar{z}\bar{\varphi}\varphi '}{1-|\varphi|^2}\\ 
	&+\frac{(1-|z|^2)\bar{\varphi} \varphi ''}{1-|\varphi|^2}+\frac{(1-|z|^2)(\bar{\varphi} \varphi ')^2}{(1-|\varphi|^2)^2}\\
	&=\frac{(1-|z|^2)S_\varphi}{2}+\frac{1-|z|^2}{4}\Big(\frac{\varphi ''}{\varphi '} \Big)^2- \frac{\bar{z}}{2}\frac{\varphi ''}{\varphi '}-\frac{\bar{z}\bar{\varphi}\varphi '}{1-|\varphi|^2}\\ 
	&+\frac{(1-|z|^2)\bar{\varphi} \varphi ''}{1-|\varphi|^2}+\frac{(1-|z|^2)(\bar {\varphi} \varphi ')^2}{(1-|\varphi|^2)^2}.
	\end{align*}
	Similarly,
	\begin{align*}
	\frac{\partial A_{\varphi}}{\partial \bar{z}} & =-\frac{z}{2}\frac{\varphi ''}{\varphi '} -1 - \frac{z\bar{\varphi}\varphi'}{1-|\varphi|^2} + \frac{(1-|z|^2)|\varphi '|^2}{(1-|\varphi|^2)^2}.
	\end{align*}
	Therefore
	\begin{equation}\label{eq3.2}
	\begin{cases}
	(1-|z|^2)\dfrac{\partial A_{\varphi}}{\partial z}&=\dfrac{(1-|z|^2)^2 S_\varphi}{2}+\big (A_{\varphi} \big )^2+\bar{z}A_{\varphi}\;,\\\\
	(1-|z|^2)\dfrac{\partial A_{\varphi}}{\partial \bar{z}}&=-zA_{\varphi} + |D_{\varphi}|^2-1.
	\end{cases}
	\end{equation}
	Since
	\begin{equation*}
	\frac{\partial^2|D_{\varphi}|}{\partial z^2}=\frac{\partial}{\partial z}\Big ( \frac{|D_{\varphi}|A_{\varphi}}{1-|z|^2}\Big)\;\;\;\text{and}\;\;\;
	\frac{\partial^2|D_{\varphi}|}{\partial z\partial \bar{z}}=\frac{\partial}{\partial z}\Big ( \frac{|D_{\varphi}|\overline{A_{\varphi}}}{1-|z|^2}\Big),
	\end{equation*}
	it follows from \eqref{eq3.2} that
	\begin{equation}\label{eq3.3}
	\begin{cases}
	\dfrac{\partial^2|D_{\varphi}|}{\partial z^2}&=\dfrac{|D_{\varphi}|}{(1-|z|^2)^2}\Big ( \dfrac{(1-|z|^2)^2 S_\varphi}{2}+2(A_{\varphi})^2 + 2\bar{z}A_{\varphi}\Big),\\
	\dfrac{\partial^2|D_{\varphi}|}{\partial z \partial \bar{z}}&=\dfrac{|D_{\varphi}|}{(1-|z|^2)^2}\big( -1+|D_{\varphi}|^2 + |A_{\varphi}|^2\big).
	\end{cases}
	\end{equation}
	We finally obtain from \eqref{eq3.1}  and \eqref{eq3.3} the expansion
	\begin{align}\label{eq3.4}
	|D_{\varphi}(z)|&=|D_{\varphi}(z_0)|+\frac{2|D_{\varphi}(z_0)|}{1-|z_0|^2}\operatorname{Re}\big (A_{\varphi}(z_0)(z-z_0)\big )\notag\\
	&+\frac{|D_{\varphi}(z_0)|}{(1-|z_0|^2)^2}\operatorname{Re}\Big [\Big ( \frac{(1-|z_0|^2)^2 S_\varphi(z_0)}{2}+2(A_{\varphi}(z_0))^2 + 2\bar{z_0}A_{\varphi}(z_0)\Big)(z-z_0)^2\Big ]\notag\\
	&+\frac{|D_{\varphi}(z_0)|}{(1-|z_0|^2)^2}\big( -1+|D_{\varphi}(z_0)|^2 + |A_{\varphi}(z_0)|^2\big)|z-z_0|^2+O(|z-z_0|^3).
	\end{align}
	\begin{The}\label{th3}
		Let $\varphi:\mathbb{D}\rightarrow\mathbb{D}$ be analytic and not constant, $S_{\varphi}$ the Schwarzian derivative of $\varphi$, and $z_0$ a critical point of $|D_{\varphi}(z)|$ with $\varphi'(z_0)\neq 0$.
		\begin{enumerate}
			\item[(i)] If $|D_{\varphi}(z)|$ has a local maximum at $z_{0}$, then
			\begin{equation}\label{eq3.5}
			(1-\left\vert z_{0}\right\vert^{2})^2  \left\vert S_{\varphi}(  z_{0})  \right\vert \leq 2(1-|D_{\varphi}(z_0)|^2).
			\end{equation}
			\item[(ii)] If strict inequality holds in \eqref{eq3.5}, then $|D_{\varphi}(z)|$ has a strict local maximum at $z_0$; in particular, $z_0$ is an isolated critical point.
			\item[(iii)] If $(1-\left\vert z_{0}\right\vert^{2})^2  \left\vert S_{\varphi}(  z_{0})  \right\vert > 2(1-|D_{\varphi}(z_0)|^2)$, then $z_0$ is an isolated critical point and there exists a neighborhood $U_0$ of $z_0$ such that $U_0 \cap C_\varphi(t_0)$ consists of two arcs that intersect at $z_0$ under a positive angle, where $t_0=|D_{\varphi}(z_0)|$.
		\end{enumerate}
	\end{The}
	\begin{proof}
		By hypothesis $A_{\varphi}(z_0)=0$. It follows from \eqref{eq2.5} that
		\begin{align}\label{eq3.6}
		\frac{|D_{\varphi}(z)|}{|D_{\varphi}(z_0)|}-1=\frac{|z-z_0|^2}{(1-|z_0|^2)^2}&\Big [\operatorname{Re}\Big ( \frac{(1-|z_0|^2)^2S_{\varphi}(z_0)}{2}\frac{(z-z_0)^2}{|z-z_0|^2}\Big )\notag \\ &-1+|D_{\varphi}(z_0)|^2+O(|z-z_0|)  \Big ]
		\end{align}
		as $z\to z_0$. If  $|D_{\varphi}(z)|$ has a local maximum at $z_0$, then the term in \eqref{eq3.6} between the square brackets is $\leq 0$ near $z_0$. Hence,
		\[
		\frac{(1-|z_0|^2)^2|S_{\varphi}(z_0)|}{2}-1+|D_{\varphi}(z_0)|^2 \leq 0.
		\]
		This proves $(i)$.\\
		Since
		\begin{align*}
		&\operatorname{Re}\Big ( \frac{(1-|z_0|^2)^2S_{\varphi}(z_0)}{2}\frac{(z-z_0)^2}{|z-z_0|^2}\Big ) -1+|D_{\varphi}(z_0)|^2+O(|z-z_0|) \\
		&\leq \Big |\frac{(1-|z_0|^2)^2S_{\varphi}(z_0)}{2}\Big | -1+|D_{\varphi}(z_0)|^2+O(|z-z_0|)
		\end{align*}
		as $z\to z_0$, we can deduce now $(ii)$ from \eqref{eq3.6} and the assumption that
		\[
		|D_{\varphi}(z)|<|D_{\varphi}(z_0)|
		\]
		in a puncture neighborhood of $z_0$.\\
		Now we prove $(iii)$. From \eqref{eq3.2}, the expansion of $A_{\varphi}(z)$ around the critical point  $z_0$ gives
		\begin{align*}
		A_{\varphi}(z)&= \frac{\partial A_{\varphi}(z_0)}{\partial z}(z-z_0)+\frac{\partial A_{\varphi}(z-0)}{\partial \bar{z}}(\bar{z}-\bar{z_0})+O(|z-z_0|^2)\\
		&=\frac{1}{1-|z_0|^2}\Big[\frac{(1-|z_0|^2)^2 S_\varphi(z_0)}{2}(z-z_0)-(1-|D_{\varphi}(z_0)|^2)(\bar{z}-\bar{z_0})\Big]+O(|z-z_0|^2)\\
		&=\frac{|z-z_0|}{1-|z_0|^2}\Big[\frac{(1-|z_0|^2)^2 S_\varphi(z_0)}{2}\frac{(z-z_0)}{|z-z_0|}-(1-|D_{\varphi}(z_0)|^2)\frac{(\bar{z}-\bar{z_0})}{|z-z_0|}\Big]+O(|z-z_0|^2)
		\end{align*}
		as $z\to z_0$. Therefore
		\begin{equation*}
		\frac{(1-|z_0|^2)|A_{\varphi}(z)|}{|z-z_0|}\geq \frac{(1-|z_0|^2)^2 |S_\varphi(z_0)|}{2}-(1-|D_{\varphi}(z_0)|^2)+O(|z-z_0|)
		\end{equation*}
		in a puncture neighborhood of $z_0$. This inequality and the hypothesis show that $z_0$ is an isolated critical point.\\
		Finally, for simplicity we write $a:=\dfrac{(1-|z_0|^2)^2|S_\varphi(z_0)|}{2}$, $b:=1-|D_{\varphi}(z_0)|^2$, and $z-z_0=\rho e^{i\vartheta}$. From \eqref{eq3.6} we have
		\[
		\frac{(1-|z|^2)^2}{a\rho^2}\frac{|D_{\varphi}(z)|-|D_{\varphi}(z_0)|}{|D_{\varphi}(z_0)|}=\cos(\gamma+2\vartheta)-\frac{b}{a}+O(\rho)
		\]
		as $\rho \to 0$, where $\gamma=\operatorname{arg}S_\varphi(z_0)$.
		Define
		\begin{equation*}
		h(\rho,\vartheta):=
		\begin{cases}
		\dfrac{(1-|z|^2)^2}{a\rho^2}\dfrac{|D_{\varphi}(z)|-|D_{\varphi}(z_0)|}{|D_{\varphi}(z_0)|}&\;\; \text{if}\;\;\rho>0,\\
		\cos(\gamma+2\vartheta)-\frac{b}{a}&\;\;\text{if}\;\;\rho=0. 
		\end{cases}
		\end{equation*}
		The function $h$ is smooth. Also, for $\rho >0$, $z\in C_\varphi(t_0)$ if and only if $h(\rho, \vartheta)=0$. The $0$ level set of $h$ consists of isolated points or Jordan arcs. By assumption $0<b/a<1$. Therefore $h(0,\vartheta)=0$ when $\cos (\gamma+2\vartheta)=b/a$, which occurs at four different directions $\vartheta_1$, $\vartheta_2$, $\vartheta_3=\vartheta_1+\pi$, $\vartheta_4=\vartheta_2+\pi$. Then $\dfrac{\partial h}{\partial \vartheta}(0,\vartheta_j)=-2\sin(\gamma + 2\vartheta_j)\neq 0$. Hence, $\nabla h(0,\vartheta_j)\neq 0$ which gives that the $0$ level set of $h$ near $(0,\vartheta_j)$ is a single arc: otherwise the changes of sign in $h$ and the continuity of $\nabla h$ would imply $\nabla h(0,\vartheta_j)=0$. This finishes the proof.	
	\end{proof}
	\noindent\textbf{Remark.} In the interesting case $(1-|z_0|^2)^2 |S_\varphi(z_0)|=2(1-|D_{\varphi}(z_0)|^2)$, the critical point $z_0$ does not have to be isolated: the function $$\varphi(z):=\dfrac{(1+z)^a-(1-z)^a}{(1+z)^a+(1-z)^a},\;\;\;0<a<1,$$ (see Example 2) has critical points all over the geodesic $(-1,1)$.
	\section{Examples}\label{sec4}
	First, we give an expression for the hyperbolic curvature $\kappa (z(t))$ of any trajectory $z(t)$, $\omega^-<t<\omega ^+$, in terms of the Schwarzian derivative and the second order operator $A_\varphi$, for any non constant holomorphic function $\varphi: \mathbb{D}\to \mathbb{D}$.
	\begin{Lem}\label{lemmacurvature}
		$\kappa(z(t))=-\dfrac{|A_\varphi(z(t))|}{2}\operatorname{Im}\Big [\dfrac{(1-|z(t)|^2)^2S_\varphi(z(t))}{A_\varphi^2(z(t))}  \Big]$.
	\end{Lem}\label{lem}
	\begin{proof}
		The hyperbolic curvature of the trajectory $z(t)$ verifies the formula (see \cite{MaMi94}, p. 84)
		\begin{equation}\label{eq4.1}
			\kappa(z(t))=\frac{1-|z(t)|^2}{|z'(t)|}\operatorname{Im}\Big [\frac{z''(t)}{z'(t)} \Big]
			+ 2\operatorname{Im}\Big [\frac{\overline{z(t)}z'(t)}{|z'(t)|}  \Big].
		\end{equation}
		From the initial value problem \eqref{eq2.1} we obtain
		\begin{align}\label{eq4.2}
			(1-|z(t)|^2)\operatorname{Im}\Big [\frac{z''(t)}{z'(t)}\Big]&=-2t\operatorname{Im}\Big [z'(t)\frac{d}{dt}A_\varphi(z(t))  \Big]\\
			&=-(1-|z(t)|^2)\operatorname{Im}\Big [ \frac{1}{A_\varphi(z(t))}\frac{d}{dt}A_\varphi(z(t)) \Big].\notag
		\end{align}
		Thus, combining \eqref{eq4.1} and \eqref{eq4.2},
		\begin{equation}\label{eq4.3}
			\kappa(z(t))=2|A_\varphi(z(t))|\Big [\operatorname{Im}\Big (\frac{\overline{z(t)}}{A_\varphi(z(t))} \Big)-t\operatorname{Im}\Big ( \frac{1}{A_\varphi(z(t))}\frac{d}{dt}A_\varphi(z(t))\Big)\Big].
		\end{equation} 
		Since 
		\begin{align*}
			t\frac{d}{dt}A_\varphi(z(t))=\frac{1-|z(t)|^2}{2A_\varphi(z(t))}\frac{\partial A_\varphi}{\partial z}(z(t))+\frac{1-|z(t)|^2}{2\overline{A_\varphi(z(t))}}\frac{\partial A_\varphi}{\partial \overline{z}}(z(t)),
		\end{align*}
		
		\noindent we get from \eqref{eq3.2}
		\begin{align}\label{eq4.4}
			\frac{t}{A_\varphi(z(t))}\frac{d}{dt}A_\varphi(z(t))=
			\frac{1}{2}\Big[\dfrac{(1-|z(t)|^2)^2S_\varphi(z(t))}{2A_\varphi^2(z(t))}&+2i\operatorname{Im}\Big(\frac{\overline{z(t)}}{A_\varphi(z(t))}\Big)\\&+\frac{|D_\varphi(z(t))|^2-1}{|A_\varphi(z(t))|^2}\Big].\notag
		\end{align}
		The result now follows from \eqref{eq4.3} and \eqref{eq4.4}.
	\end{proof}
	\noindent \textbf{Example 1.}
	Consider the function $\varphi(z)=e^{-g(z)}$ with
	$g(z)=\left(  \dfrac{1+z}{1-z}\right)  ^{a}$, $0<a<1$, $z\in\mathbb{D}$. Our main goal in this example is to show that any trajectory $z(t)$ starts at a point on the boundary of $\mathbb{D}$, as $t\to 0^+$, and ends at $-1$. We will prove the following statements:
	\begin{enumerate}
		\item[(a)] $A_{\varphi}(z)\neq 0,\infty$ for all $z\in \mathbb{D}$. This proves that $z(t)$ starts and ends on $\mathbb{T}$.
		\item[(b)] For all $\zeta\in\mathbb{T},\,\zeta\neq-1$, $\lim_{z\to\zeta}|D_\varphi(z)|=0$. This proves that $z(t)$ ends at $-1$ and $\omega^-=0$ since $t=|D_\varphi (z(t))|$.
		\item[(c)] The trajectory through the origin is the interval $(-1,1)$ starting at $1$ $(\omega^-=0)$ and ending at $-1$ $(\omega^+=a)$.
		\item[(d)] The level sets of $|D_\varphi|$ are open Jordan arcs, being $-1$ its end points.
		\item[(e)] For $\Gamma^-$ defined as in \eqref{eq2.4}, $\overline{\Gamma^-}\setminus \Gamma^-$ reduces to a single point.
		\item[(f)] $\omega^+=a$. This proves that the trajectory ends at $-1$ tangentially to the real axis.
	\end{enumerate}  
	(a) Direct computations from \eqref{eq1.2} give
	\begin{align}\label{eq4.5}
		A_{\varphi}(z)&=\frac{1-|z|^2}{1-z^2}\Big[a+z-a\Big (\frac{1+z}{1-z}\Big )^{a}\coth \Big ( \operatorname{Re}\Big (\frac{1+z}{1-z}\Big )^{a}\Big )\Big]-\overline{z}\notag\\
		&=\frac{a\left(  1-\left\vert z\right\vert ^{2}\right)  +2i\operatorname{Im}%
			z}{1-z^{2}}-a\frac{1-\left\vert z\right\vert ^{2}}{1-z^{2}}\left(\frac	{1+z}{1-z}\right)^{a}\coth \Big [ \operatorname{Re}\Big (\frac{1+z}{1-z}\Big )^{a}\Big ].
	\end{align}
	Write $z=\dfrac{\xi+iy}{1+i\xi y}$, $-1<\xi,y<1$, then%
	\[
	\frac{1+z}{1-z}=\frac{1+\xi}{1-\xi}\frac{1+iy}{1-iy},\;\;
	z=\frac{\xi\left(  1+y^{2}\right)  +iy\left(  1-\xi^{2}\right)  }%
	{1+\xi^{2}y^{2}},\;\;\;1-\left\vert z\right\vert ^{2}=\frac{\left(
		1-\xi^{2}\right)  \left(  1-y^{2}\right)  }{1+\xi^{2}y^{2}}.
	\]
	Using \eqref{eq4.5} we get
	\begin{align}\label{eq4.6}
		\frac{1-z^{2}}{a(  1-\left\vert z\right\vert ^{2})  }%
		A_{\varphi}(z) & =1+\frac{2i\operatorname{Im}z}{1-|z| ^{2}}-\Big (  \frac{1+z}{1-z}\Big )  ^a \coth \Big [
		\operatorname{Re}\Big (\frac{1+z}{1-z}\Big )^{a}\Big ] \notag \\
		& =1+\frac{2iy}{1-y^{2}}-\Big (\frac{1+\xi}{1-\xi}\Big )  ^{a}\Big [
		\cos ( 2a\tan^{-1}y )  +i\sin ( 2a\tan^{-1}y )\Big ]\\
		&\cdot \coth\Big [ \Big (  \frac{1+\xi}{1-\xi}\Big )^{a}\cos (  2a\tan^{-1}y )\Big ] .\notag
	\end{align}
	Since $\cos ( 2a\tan^{-1}y )\neq 0$ then 
	\begin{equation*}
		\operatorname{Re}\Big[\frac{1-z^{2}}{a(  1-\left\vert z\right\vert ^{2})  }%
		A_{\varphi}(z)\Big ]=0\iff 
		\Big (  \frac{1+\xi}{1-\xi}\Big )^{a}\coth\Big [ \Big (  \frac{1+\xi}{1-\xi}\Big )^{a}\cos (  2a\tan^{-1}y )\Big ]
		=\frac{1}{\cos ( 2a\tan^{-1}y )}.
	\end{equation*}
	Hence
	\begin{align*}
		\frac{1-z^{2}}{a(  1-\left\vert z\right\vert ^{2})  }%
		A_{\varphi}(z)=0 & \iff
		\frac{2y}{a(1-y^{2})}=\tan (  2a\tan^{-1}y )\iff y=0\\
		& \iff \Big (  \frac{1+\xi}{1-\xi}\Big )^{a}= \tanh \Big (  \frac{1+\xi}{1-\xi}\Big )^{a},
	\end{align*}
	which is impossible since $t>\tanh t$ for $t>0$. This shows that $A_{\varphi}(z)\neq0$ for all $z\in \mathbb{D}$.\\\\
	(b) Direct computations give
	\begin{equation*}
		\frac{1-|z|^2}{|1-z^2|}=\frac{1-y^2}{1+y^2}=\cos ( 2\tan^{-1}y ).
	\end{equation*}
	Then
	\begin{align}\label{eq4.7}
		|D_{\varphi}(z)|&=a\frac{1-|z|^2}{|1-z^2|}\frac{|g(z)|}{\sinh\operatorname{Re}(g(z))}\notag\\
		&= a\;\frac{\Big(\frac{1+\xi}{1-\xi}\Big)^{a}\cos ( 2\tan^{-1}y )}{\sinh\Big[\Big(\frac{1+\xi}{1-\xi}\Big)^{a}\cos ( 2a\tan^{-1}y )\Big]}
		=:|D_\varphi(\xi,y)|\\ &\leq a\;\frac{\Big(\frac{1+\xi}{1-\xi}\Big)^{a}\cos ( 2\tan^{-1}y )}{\sinh\Big[\Big(\frac{1+\xi}{1-\xi}\Big)^{a}\cos ( 2\tan^{-1}y )\Big]}\leq a.\notag
	\end{align}
	If $-1\neq \zeta \in \mathbb{T}$, and $z\to \zeta$, then $\Big(\dfrac{1+\xi}{1-\xi}\Big)^{a}$ remains bounded and $y\to \pm 1$. Therefore the result follows from the second equality in \eqref{eq4.7}.\\\\
	(c) We see from \eqref{eq4.5} that $A_\varphi(z)$ is real for $z$ real. Then, by the initial value problem \eqref{eq2.3}, the tangent vector to the trajectory through any point of the real axis is horizontal. Hence, $(-1,1)$ is a trajectory. Also, taking $y=0$ in \eqref{eq4.7} we obtain $\omega^+=\lim_{\xi\to-1}|D_\varphi(\xi,0)|=a$.\\\\
	(d) Let $J$ be a component of $C_\varphi(t)$. Since $|D_\varphi|$ has no critical points, $J$ does not reduce to a single point. Also, by part (a), $J$ is not a closed curve inside $\mathbb{D}$. Furthermore, by part (b), $J$ is a Jordan arc that starts and ends at $-1$; so, for $\xi\in(-1,1)$ sufficiently close to $-1$, the hyperbolic geodesic $\lambda_\xi$ orthogonal to the real axis through the point $\xi$ meets $J$ at more than one point. On the other hand, by part (c), there is a unique point  $x\in (-1,1)$ belonging to $C_\varphi(t)$. We will show that $x\in J$. If this were not true, then by the symmetry of $|D_\varphi(z)|$ with respect to the real axis (see \eqref{eq4.7}), we may assume that $J$ lies entirely in the upper half plane. However, for fixed $\xi$ sufficiently close to $-1$, the expression of $|D_\varphi(z)|$ in \eqref{eq4.7} is strictly decreasing with respect to $y>0$, which contradicts the fact that $J$ meets $\lambda_\xi$ at least twice. \\\\
	(e) Since $\operatorname{Re}A_\varphi(iy)$ in \eqref{eq4.6} is negative for all $y$, the tangent vectors $z'(t)$ at points of the imaginary axis have negative real part. Hence, $\overline{\Gamma^-}\setminus \Gamma^-$ lies entirely in the closure of one of the half-planes determined by the imaginary axis.\\
	Now we reason by contradiction assuming the existence of at least two different points in $\overline{\Gamma^-}\setminus \Gamma^-$. By the previous paragraph, we may assume that there is $\zeta\in\overline{\Gamma^-}\setminus \Gamma^-$, $\zeta\neq-1,1$, and a sequence $t_n$ converging to $\omega^-=0$, such that $z(t_n)\to\zeta$ and $\kappa(z(t_n))\to \infty$. Since the Schwarzian derivative of $\varphi$ is 
	$$S_\varphi(z)=\frac{2(1-a^2-a^2g(z)^2)}{(1-z^2)^2}$$ then, by Lemma \ref{lemmacurvature} and \eqref{eq4.5}, 
	\begin{align}\label{eq4.8}
		\kappa(z(t))&=-\Bigg |\frac{(1-|z(t)|^2)\Big[a+z(t)-a\Big (  \frac{1+z(t)}{1-z(t)}\Big )  ^a\coth \Big( \operatorname{Re}\Big (  \frac{1+z(t)}{1-z(t)}\Big )  ^a\Big )\Big]-\overline{z(t)}(1-z^2(t))}{|1-z^2(t)|}\Bigg |\notag\\
		&\cdot\operatorname{Im}\frac{(1-|z(t)|^2)^2\Big(1-a^2-a^2\Big (  \frac{1+z(t)}{1-z(t)}\Big )  ^{2a}\Big)}{\Big\{(1-|z(t)|^2)\Big[a+z(t)-a\Big (  \frac{1+z(t)}{1-z(t)}\Big )  ^a\coth \Big( \operatorname{Re}\Big (  \frac{1+z(t)}{1-z(t)}\Big )  ^a\Big )\Big]-\overline{z(t)}(1-z^2(t))\Big\}^2}.
	\end{align}
	Therefore $\lim_{n\to \infty}\kappa(z(t_n))=0$, a contradiction.\\\\
	(f) Assume that $\omega^+<a$ and let $C:=C_\varphi(\omega^+)$. Then the trajectory $\Gamma$ is not the interval $(-1,1)$ and, by the symmetry of $|D_\varphi|$ with respect to the real axis, we may also assume that $\Gamma$ lies entirely in upper half plane. By part (d) the level curves are nested. From \eqref{eq4.7} we see that, for fixed $0<y<1$,
	\[
	\lim_{\xi\to-1}|D_\varphi(\xi,y)|=\frac{a(1-y^2)}{1+y^2}\frac{1}{\cos(2a\tan^{-1} y)},
	\]
	which takes all values once in the interval $(0,a)$. Hence, there is a unique $y\in(0,1)$ such that $\lim_{\xi\to-1}|D_\varphi(\xi,y)|=\omega^+$. Therefore $\Gamma$ must be tangent at $-1$ to the circle through $iy$ and $\pm 1$, which makes an angle $\beta_0$ at $-1$ with the interval $(-1,1)$. Write $1+z(t)=r(t)e^{i\beta(t)}$. Then, as $t\to\omega^+$,
	\[
	-\frac{\overline{z(t)}(1-z(t)^2)}{1-|z(t)|^2}\to 1+i\tan\beta_0,
	\]
	and
	\begin{align*}
	\Big (  \frac{1+z(t)}{1-z(t)}\Big )  ^a\coth \Big( \operatorname{Re}\Big (  \frac{1+z(t)}{1-z(t)}\Big )  ^a\Big )&=\frac{\operatorname{Re}\Big (  \frac{1+z(t)}{1-z(t)}\Big )^{a}}{\tanh\Big[\operatorname{Re}\Big (  \frac{1+z(t)}{1-z(t)}\Big )^{a}\Big]}\Bigg[1+i\frac{\operatorname{Im}\Big (  \frac{1+z(t)}{1-z(t)}\Big )^{a}}{\operatorname{Re}\Big (  \frac{1+z(t)}{1-z(t)}\Big )^{a}}\Bigg]\\
	&\to 1+i\tan(a\beta_0).
	\end{align*}
	Hence, by \eqref{eq4.8}, $\kappa(z(t))\to 0$ as $t\to\omega^{+}$, which means that $\beta_0=0$, a contradiction. Therefore $\omega^{+}=a$.\\\\   
	\noindent \textbf{Example 2.}
	Consider the function $\varphi (z)=\dfrac{g(z)-1}{g(z)+1}$ with $g(z)=\left(  \dfrac{1+z}{1-z}\right)  ^{a}$, $0< a < 1$, $z\in\mathbb{D}$. The function $\varphi$ maps $\mathbb{D}$ conformally onto the symmetric lens-shaped domain between two circular arcs that meet $\pm 1$ under the angle $\pi a$. As in the previous example we write $z=\dfrac{\xi+iy}{1+i\xi y}$, $-1<\xi,y<1$, then%
	\begin{align*}
		|D_{\varphi}(z)|=a\frac{1-|z|^2}{|1-z^2|}\frac{|g(z)|}{|\operatorname{Re} g(z)|}=a\frac{1-y^2}{1+y^2}\frac{1}{\cos (2a\tan^{-1}y)}=a\frac{\cos (2\tan^{-1}y)}{\cos (2a\tan^{-1}y)}\leq a,
	\end{align*}
	with equality if and only if $y=0$. Hence, $|D_{\varphi}(z)|$ attains its maximum value $a$ at all points on the interval $(-1,1)$. Furthermore, the level set $C_{\varphi}(t)$, for $0 < t < a$, consists of two circular arcs symmetric about the real axis that meet $\pm 1$. Also, 
	\[
	(1-|z|^2)^2|S_{\varphi}(z)|=|S_{\varphi}(0)|=2(1-a^2)=2(1-|D_{\varphi}(z)|^2),\;\;\text{for all}\;\; -1 < z < 1; 
	\]
	and
	\[
	A_{\varphi}(z)=\frac{1-|z|^2}{1-z^2}\Big[a-ag(z)\frac{\operatorname{Re}g(z)}{|\operatorname{Re}(g(z))|^2}+z\Big]-\overline{z},
	\]
	which vanishes for all $-1<z<1$, as expected.\\\\
	\textbf{Example 3.} (See Example 5.3 in \cite{MePo00}) Let $0<\theta<\pi/2$, $a=K(\cos\theta)$ and $b=K(\sin\theta)$, where $K(\kappa)$ denotes the complete elliptic integral of the first kind; and consider $\varphi=h\circ g$, where
	\[
	h(w)=i\frac{1-e^{iw\pi/a}}{1+e^{iw\pi/a}}\;\;\text{and}\;\;g(z)=\int_{0}^{z}(1-2\cos(2\theta)\zeta^2+\zeta^4)^{-1/2}d\zeta\,.
	\]
	The function $\varphi$ maps $\mathbb{D}$ conformally onto the domain in $\mathbb{D}$ that is symmetric with respect to the real and imaginary axis, and bounded by the orthocircles through
	\[
	\pm i\big(1-e^{-\pi b/(2a)}\big)/\big(1+e^{-\pi b/(2a)}\big)
	\]
	and by two arcs of $\mathbb{T}$. The computations show that
	\begin{equation}\label{eq4.9}
		|D_\varphi(z)|=\frac{\pi}{2a}\frac{(1-|z|^2)|g'(z)|}{\cos\big(\frac{\pi}{a}\operatorname{Re}g(z)\big)}
	\end{equation}
	and
	\begin{equation}\label{eq4.10}
		S_\varphi(z)=2\frac{(c+\alpha^2)+(c^2-2\alpha c-3)z^2+(c+\alpha^2)z^4}{(1-2cz^2+z^4)^2}\,,
	\end{equation}
	where $\alpha=\pi/(2a)$ and $c=\cos(2\theta)$.
	It is not difficult to show from \eqref{eq4.9} that 
	\[
	\dfrac{\partial|D_\varphi(x)|}{\partial x}\Big |_{x=0}=0\;\;\;\text{and}\;\;\;\dfrac{\partial|D_\varphi(iy)|}{\partial y}\Big |_{y=0}=0\,;
	\]
	hence, $z=0$ is a critical point of $|D_\varphi|$. We also see from \eqref{eq4.10} that $|S_\varphi (0)|=2\Big(\dfrac{\pi^2}{4a^2}+\cos(2\theta)\Big)$. Numerical evidence shows that $\dfrac{\pi^2}{2a^2}+\cos(2\theta)>1$, which is equivalent to the inequality
	\[
	|S_\varphi(0)|>2\Big(1-|D_\varphi(0)|^2\Big).
	\]
	Therefore, by Theorem \ref{th3}, part $(iii)$, $z=0$ is a saddle point of $|D_\varphi|$. \\\\
	\textbf{Example 4.} Consider the Blaschke product 
	\begin{align*}
		\varphi (z) = \frac{z^{2}-c^{2}}{1-c^{2}z^{2}}\,,\,\,  0<c<1.
	\end{align*}
	Then
	\begin{align*}
		\frac{1-|z|^2}{1-|\varphi(z)|^2}=\frac{|1-c^{2}z^{2}|^{2}}{(1-c^{4})(1+|z|^{2})},\,\,\,
		\varphi '(z) = \frac{2(1-c^{4})z}{(1-c^{2}z^{2})^{2}}.
	\end{align*}
	Therefore
	\begin{equation*}
		|D_\varphi(z)|=\frac{2|z|}{1+|z|^2},\;\;A_\varphi(z)=(1-|z|^2)\frac{\partial\log |D_\varphi(z)|}{\partial z}=\frac{(1-|z|^2)^2}{2z(1+|z|^2)}.
	\end{equation*}
	In this example the level curves are the circles centered at the origin, and the trajectories have the parametrization $z(t)=\dfrac{te^{i\vartheta}}{1+\sqrt{1-t^2}}$, $0<t<1$ with $\vartheta$ fixed.\\\\

\end{document}